\documentclass[12pt]{amsart}
\usepackage{amsfonts,amsthm,latexsym,amsmath,amssymb,amscd,amsmath, mathrsfs, url, cite, graphicx }
\usepackage[pdftex]{color}
\usepackage[bookmarks=true,hyperindex,pdftex,colorlinks, citecolor=blue,linkcolor=blue, urlcolor=blue]{hyperref}
 \newcommand{\sgn}{\operatorname{sgn}}
\newcounter{count}
\numberwithin{count}{section}
\newtheorem{Lemma}[count]{Lemma}
\newtheorem{Remark}[count]{Remark}
\newtheorem{Example}[count]{Example}
\newtheorem{Definition}[count]{Definition}
\newtheorem{theorem}[count]{Theorem}
\newtheorem{Corollary}[count]{Corollary}
\newtheorem{Theorem}[count]{Theorem}
\newtheorem{Conjecture}[count]{Conjecture}
\newtheorem{Problem}[count]{Open Problem}
\newtheorem{Statement}[count]{Statement}

\begin{document}

\author[A.~Vishnyakova]{Anna Vishnyakova}

\address{Department of Mathematics, Holon Institute of Technology,
Israel}
\email{annalyticity@gmail.com}

\title[Polynomially Deformed Normalized Pochhammer Sequences ]
{Polynomially Deformed Normalized Pochhammer Sequences Having 
Generating Functions With Only Real Non-positive Zeros}
\maketitle

\begin{abstract}

For a given real number $a>0$ 
and a given real polynomial $P_n\in \mathbb{R}[x]$ of degree $n=0, 1, 2, \ldots$
it is easy to see that  $ \sum_{k=0}^\infty \frac{(a)_k}{k!} P_n(k) z^k =\frac{S_{n, a}(z)}{(1-z)^{a+ n}}, 
\   |z|<1,  $ where $S_{n, a}$ is a real polynomial of degree not greater than $n.$
Here $(a)_k =a(a+1)\cdot \ldots \cdot (a+k-1),\    (a)_0 = 1,$ is the  rising factorial, or the 
Pochhammer symbol. We consider the following open problem: 
to describe the set of real polynomials $P_n\in \mathbb{R}[x]$ 
of degree $n=0, 1, 2, \ldots,$ such that the corresponding polynomial $S_{n, a}$
has all real non-positive zeros. In the case $a=1$ this problem has been studied 
in \cite{vish}. We establish several new necessary conditions 
and  several sufficient conditions, present a number of  important 
examples, and formulate several open problems.

\end{abstract}



\section{Introduction}

\begin{Definition}
 A sequence of nonnegative numbers $(a_k)_{k=0}^\infty$
is called  totally positive if all minors of the
infinite matrix

\begin {equation}
\label{mat}
 \left(
  \begin{array}{ccccc}
   a_0 & a_1 & a_2 & a_3 &\ldots \\
   0   & a_0 & a_1 & a_2 &\ldots \\
   0   &  0  & a_0 & a_1 &\ldots \\
   0   &  0  &  0  & a_0 &\ldots \\
   \vdots&\vdots&\vdots&\vdots&\ddots
  \end{array}
 \right)
\end {equation}
 are non-negative. 
\end{Definition}

The class of totally positive sequences is denoted by $\mathrm{TP}.$
The class of generating functions of totally positive sequences
$f(x)=\sum_{k=0}^\infty a_k x^k$
is  denoted by  $\widetilde{\mathrm{TP}}.$

The concept of total positivity has found numerous applications and
has been studied from many different perspectives (see, for example, 
T.~Ando, \cite{ando}, S.~Karlin,  \cite{tp},  or A.~Pinkus, \cite{pin}).  It has 
applications  in the distribution of zeros of polynomials and entire functions,  
P\'olya frequency sequences, unimodality and log-concavity,  stochastic 
processes and approximation theory, mechanical systems,  planar  
networks, combinatorics, and representation theory.

Checking whether all minors of an infinite matrix are nonnegative is
often a very difficult task. The paper \cite{kv} contains an easily verifiable 
sufficient condition.

The class $\widetilde{\mathrm{TP}}$ was completely characterized by the
classical theorem by M.~Aissen, A.~Edrei, I.J.~Schoenberg,  and A.~Whitney.

{\bf  Theorem A} (M.~Aissen, A.~Edrei, I.J.~Schoenberg, A.~Whitney, \cite{aissen}).
{ \it  Let $(a_k)_{k=0}^\infty$  be a given sequence of nonnegative numbers.  
Then $f(z)=\sum_{k=0}^\infty a_kz^k \in \widetilde{\mathrm{TP}}$ if and
only if
$$f(z)=C z^n e^{\gamma z}\prod_{k=1}^\infty (1+\alpha_kz)/(1-\beta_kz),$$
where $C\ge 0, n\in \mathbb{ N}\cup\{0\},
\gamma\ge0,\alpha_k\ge0,\beta_k\ge0,\sum(\alpha_k+\beta_k)
<\infty.$}

The following fact is an immediate corollary of Theorem~A.

\begin{Corollary}  A polynomial with nonnegative coefficients 
$P(x) =\sum _{k=0}^n a_k x^k $   has only real zeros if and only 
if the sequence of its coefficients is totally positive: $(a_0,
a_1, \ldots , a_n, 0, 0, \ldots ) \in \mathrm{TP}.$
\end{Corollary}

In \cite{vish},  the following problem was studied: for which real polynomials
$P$ is the sequence $(P(k))_{k=0}^\infty$ totally positive? Using Theorem~A, we 
can equivalently reformulate this problem as follows. It is easy to see that,
for a real polynomial $P_n$ of degree $n=0, 1, 2, \ldots,$ we have
\begin{equation}
\label{ll1}
\sum_{k=0}^\infty P_n(k) z^k =\frac{S_n(z)}{(1-z)^{n+1}},\    |z|<1, 
\end{equation}
where $S_n$ is a real polynomial of degree not greater than $n.$ The
sequence $(P_n(k))_{k=0}^\infty,  P_n \not\equiv 0,$ is totally positive if and only if the
polynomial $S_n$ has all real non-positive zeros and positive leading 
coefficient.

In this paper, we will study a generalization of the above problem, namely, 
we are interested in  the following open problem. For a given real number $a>0$ 
and a given real polynomial $P_n\in \mathbb{R}[x]$ of degree $n=0, 1, 2, \ldots,$
it is easy to see that
\begin{eqnarray}
\label{ll2}
& \sum_{k=0}^\infty \frac{(a)_k}{k!} P_n(k) z^k =\frac{S_{n, a}(z)}{(1-z)^{a+ n}}, 
\\ \nonumber &
\   |z|<1, \   \arg(1-x)=0 \  \mbox{for} \   0 \leq x <1,
\end{eqnarray}
where $S_{n, a}$ is a real polynomial of degree not greater than $n$
(we  prove this formula below). Here, as usual, $(a)_k =a(a+1)\cdot \ldots \cdot (a+k-1), 
(a)_0 = 1,$  is the  rising factorial, or the 
Pochhammer symbol. The following open problem is of interest to us.

\begin{Problem}
\label{pr1}
Given a real number $a>0,$  describe the set of real polynomials $P_n\in \mathbb{R}[x]$ 
 of degree $n=0, 1, 2, \ldots$ such that the corresponding polynomial $S_{n, a}$ in (\ref{ll2})
 has all real non-positive zeros.
\end{Problem}

We consider the following simple Taylor series identity
$$ \sum_{k=0}^\infty \frac{(a)_k}{k!} z^k = \frac{1}{(1-z)^a}, $$
$$ |z| <1,  \   \arg(1-x)=0 \  \mbox{for} \   0 \leq x <1.  $$ 
For $m\in \mathbb{N}\cup \{0\},$ differentiate this identity $m$ times
and multiply both sides by  $\frac{x^m}{a(a+1)\cdot \ldots \cdot (a+m-1)}.$  We obtain
\begin{equation}
\label{f1}
\sum_{k=0}^\infty \frac{(a)_k}{k!} \cdot \frac{k(k-1)\cdot \ldots \cdot (k-m+1)}{a(a+1)\cdot \ldots \cdot (a+m-1)} z^k
=\frac{z^m}{(1-z)^{a+m}}.
\end{equation}

Let $P_n\in \mathbb{R}[x]$ be a real polynomial of degree $n\in \mathbb{N}.$ 
Let us decompose this polynomial in the form
\begin{eqnarray}
\label{f2}
&
P(x)= a_0 + a_1 \frac{x}{a} +a_2 \frac{x(x-1)}{a(a+1)} +a_3 \frac{x(x-1)(x-2)}{a(a+1)(a+2)}\\
\nonumber &
+  \ldots +a_n \frac{x(x-1)\cdot \ldots \cdot (x-n+1)}{a(a+1)\cdot \ldots \cdot (a+n-1)}.
\end{eqnarray}
Formula  (\ref{f1}) yields 
\begin{eqnarray}
\label{f3}
&
\sum_{k=0}^\infty P(k) x^k= \frac{a_0}{(1-x)^a} +  \frac{a_1 x}{(1-x)^{a+1}} 
+ \frac{a_2 x^2}{(1-x)^{a+2}} + \ldots + \frac{a_n x^n}{(1-x)^{a+n}}\\
\nonumber &
=\frac{1}{(1-x)^a} \left(a_0 + a_1 \cdot \frac{x}{1-x} + a_2 \cdot \frac{x^2}{(1-x)^2} +
\ldots +  a_n \cdot \frac{x^n}{(1-x)^n} \right) \\
\nonumber &
=\frac{1}{(1-x)^{a+n}} \left(a_0 (1-x)^n + a_1 x(1-x)^{n-1} + a_2 x^2(1-x)^{n-2} +
\ldots +  a_nx^n\right).
\end{eqnarray}
Thus, 
$$S_{n, a}(x)=a_0 (1-x)^n + a_1 x(1-x)^{n-1} + a_2 x^2(1-x)^{n-2} +
\ldots +  a_n x^n,$$ 
and we will study  under which assumptions on $P_n$ all zeros
of $S_{n, a}$ are real and non-positive.

Denote by 
$$Q_{n, a}(x) =a_0 +a_1 x +a_2 x^2 + \ldots + a_n x^n.$$
We have
\begin{equation}
\label{ff1}
S_{n, a}(x) =(1-x)^n Q_{n, a}\left(\frac{x}{1-x}\right).
\end{equation}
Thus, all zeros of the polynomial $S_{n, a}$ are real and non-positive if and only if all zeros of the 
polynomial $Q_{n, a}$ are real and belong to the segment $[-1, 0]$ (we note that the possible zero 
of $Q_{n, a}$ at the point $-1$ does not correspond to any zero of $S_{n, a}$).
Motivated by this observation, we define the following linear operator.

\begin{Definition} Let $\mathcal{L}_a$  denote the linear operator  $\mathcal{L}_a : \mathbb{R}[x] \rightarrow 
\mathbb{R}[x]$ such that
\begin{eqnarray}
\label{f4}
&  \mathcal{L}_a \left(a_0 + a_1 \frac{x}{a} +a_2 \frac{x(x-1)}{a(a+1)} +a_3 \frac{x(x-1)(x-2)}{a(a+1)(a+2)}
+  \ldots +\right.
\\  \nonumber &
\left. a_n \frac{x(x-1)\cdot \ldots \cdot (x-n+1)}{a(a+1)\cdot \ldots \cdot (a+n-1)}\right)
= a_0 +a_1 x +a_2 x^2 + \ldots + a_n x^n.
\end{eqnarray}
\end{Definition}
The previous discussion shows that  Open problem~\ref{pr1}  is equivalent to 
the following  problem.
\begin{Problem}
\label{pr2}
To describe the set of real polynomials $P\in \mathbb{R}[x]$  such that  the 
polynomials $\mathcal{L}_a(P)$ have only real zeros lying
in the segment $[-1, 0].$
\end{Problem}

\begin{Remark}
\label{r1}
Suppose that a polynomial 
$$P(x)= a_0 + a_1 \frac{x}{a} +a_2 \frac{x(x-1)}{a(a+1)} +a_3 \frac{x(x-1)(x-2)}{a(a+1)(a+2)}
+  \ldots +  $$
$$ a_n \frac{x(x-1)\cdot \ldots \cdot (x-n+1)}{a(a+1)\cdot \ldots \cdot (a+n-1)}$$
is such that
$$\mathcal{L}_a(P)(x)
= a_0 +a_1 x +a_2 x^2 + \ldots + a_n x^n$$ 
has only real zeros lying in the segment $[-1, 0].$ Then all
coefficients of $\mathcal{L}_a(P)$ have the same sign. If $a_n >0,$
then there exists $l=0, 1, 2, \ldots, n$ such that $a_j=0$ for all
$j < l,$ and $a_j >0$ for $j=l, l+1, \ldots, n.$

\end{Remark}

The following theorem by F.~Brenti provides a simple sufficient condition for the case $a=1.$

{\bf Theorem B} (F.\,Brenti, \cite{BrentiMemoir}, see also \cite{BrentiContMath}, c.f. \cite[Lemma 1]{kav}). {\it  Let 
$P\in \mathbb{R}[x]$ be a polynomial with only real zeros, and let $\lambda(P), \Lambda(P)$ 
be the smallest and the largest zeros of $P.$ Suppose that $P(x)=0$ for all $x\in 
\left(\left[\lambda(P), -1\right] \cup \left[0, \Lambda(P)\right]\right) \cap \mathbb{Z}.$  
Then the polynomial $\mathcal{L}_1(P)$ has only real zeros 
lying in the segment $[-1, 0].$ }

In \cite{kav}, the following generalization of  F.~Brenti's result for $a>0$ and for polynomials 
with only real non-positive zeros was obtained.

{\bf Theorem C} (Dmitrii Karp and Anna Vishnyakova, \cite[Theorem 1]{kav}). {\it  Let $a>0,$
and let  $P\in \mathbb{R}[x]$ be a polynomial with only real non-positive zeros.  Let $\lambda(P)$ 
be the smallest  zero of $P.$ Suppose that $P(x)=0$ for all $x\in 
\left(\left[\lambda(P), -a\right] \cap \{ -a, -a-1, -a-2, \ldots  \}\right).$  
Then the polynomial $\mathcal{L}_a(P)$ has only real zeros 
lying in the segment $[-1, 0].$ }

Our first result is the following generalization of  F.~Brenti's result valid for every $a>0$ 
and  for polynomials with only real  zeros. 

\begin{Theorem}
\label{Th1.0}
Let $a>0,$ and $P\in \mathbb{R}[x]$ be a polynomial with only real zeros, and let 
$\lambda(P), \Lambda(P)$ be the smallest and the largest zeros of $P.$  Suppose 
that $P(x)=0$ for all $x\in \left(\left[\lambda(P), -a\right] \cap \{ -a, -a-1, -a-2, \ldots  \}\right)
\cup \left(\left[0, \Lambda(P)\right] \cap \mathbb{Z}\right).$  
Then the polynomial $\mathcal{L}_a(P)$ has only real zeros 
lying in the segment $[-1, 0].$ 
\end{Theorem}

Later we will study how close the sufficient conditions in the previous theorem are 
to being necessary.

The next theorem, due to D. G. Wagner, shows that the set of polynomials $P$
such that $\mathcal{L}_1(P)$ has only real zeros lying in the segment $[-1, 0]$
is closed under multiplication.

{\bf Theorem D} (D.G.\,Wagner, \cite{wagner}). { \it  Let 
$P_1, P_2 \in \mathbb{R}[x]$ be real polynomials such that both $\mathcal{L}_1(P_1)$ and 
$\mathcal{L}_1(P_2)$ have only real zeros lying in the segment $[-1, 0].$ Then the
polynomial $\mathcal{L}_1(P_1\cdot P_2)$ also has only real zeros 
lying in the segment $[-1, 0].$ }

\begin{Remark}   It is known that the set of $\mathrm{TP}$-sequences is not closed under
term-by-term multiplication. For further details about term-by-term 
multiplication of totally positive sequences,  see \cite{kavish}.
\end{Remark}

The question of whether an analogous to Theorem~D result holds for all $a>0$ is open.
\begin{Problem}
\label{pr3.1}
 Let $a>0$ and $P_1, P_2 \in \mathbb{R}[x]$ be real polynomials such that both 
 $\mathcal{L}_a(P_1)$ and  $\mathcal{L}_a(P_2)$ have only real zeros lying in the segment 
 $[-1, 0].$ Is it true that the polynomial $\mathcal{L}_a(P_1\cdot P_2)$  has only real zeros 
lying in the segment $[-1, 0]?$ 
\end{Problem}

The following open problems connected with the Problem~\ref{pr2} are also of interest.

\begin{Problem}
\label{pr3}
Given $a>0,$  describe the set of real polynomials $P\in \mathbb{R}[x]$ with  
only real non-positive zeros such that  the polynomials $\mathcal{L}_a(P)$
have only real zeros lying in the segment $[-1, 0].$
\end{Problem}

\begin{Problem}
\label{pr3}
Given $a>0,$  describe the set of real polynomials $P\in \mathbb{R}[x]$ with  
only real non-negative zeros such that  the polynomials $\mathcal{L}_a(P)$
have only real zeros lying in the segment $[-1, 0].$
\end{Problem}

\begin{Problem}
\label{pr3a}
Given $a>0,$ describe the set of real polynomials $P\in \mathbb{R}[x]$ 
such that  the polynomials $\mathcal{L}_a(P)$
have only real negative zeros.
\end{Problem}

\section{Some useful formulas}

1. We start with the classical Newton formulas for finite differences.

Suppose that $P(x)= a_0 + a_1 \frac{x}{a} +a_2 \frac{x(x-1)}{a(a+1)} +a_3 \frac{x(x-1)(x-2)}{a(a+1)(a+2)}
+  \ldots +  a_n \frac{x(x-1)\cdot \ldots \cdot (x-n+1)}{a(a+1)\cdot \ldots \cdot (a+n-1)}.$ We consider
the following linear operator
\begin{equation}
\label{fd}
\Delta:  \mathbb{R}[x] \rightarrow \mathbb{R}[x], \quad (\Delta(P))(x) = P(x+1) -P(x).
\end{equation}
Then the following well-known formulas hold (and can be easily checked)
\begin{eqnarray}
\label{fd1}
&   a_0 = P(0); \\ \nonumber &
a_k =\frac{(a)_k}{k!} \Delta^k(P)(0) = \frac{(a)_k}{k!} \sum_{j=0}^k (-1)^j   C_k^j P(k-j),  k=1, 2, \ldots, n.
\end{eqnarray}

2.  For every $n \in \mathbb{N},$ the following identity holds 
\begin{eqnarray}
\label{f6}
& \sum_{k=0}^n C_n^k \   \frac{x(x-1)\cdot \ldots \cdot (x-k+1)}{(a)_k}=
\\ \nonumber & \frac{(x+a)(x+a+1) \cdot 
\ldots \cdot (x+a +n -1)}{(a)_n}
\end{eqnarray}
and can be easily proved by induction.

3. Let $P(x)= a_0 + a_1 \frac{x}{a} +a_2 \frac{x(x-1)}{a(a+1)} +  \ldots +  
a_n \frac{x(x-1)\cdot \ldots \cdot (x-n+1)}{a(a+1)\cdot \ldots \cdot (a+n-1)},$ 
and $Q(x)= \mathcal{L}_a(P)(x) = a_0 +a_1 x +a_2 x^2 + \ldots + a_n x^n.$ We have
$$
  Q(-(x+1)) = \sum_{k=0}^n (-1)^k a_k \sum_{j=0}^k C_k^j x^j =: \tilde{Q} (x), 
$$
thus
$$ \mathcal{L}_a^{-1}(\tilde{Q} )(x) = \sum_{k=0}^n (-1)^k a_k \sum_{j=0}^k C_k^j  
\frac{x(x-1)\cdot \ldots \cdot (x-j+1)}{(a)_j}.  $$
Using (\ref{f6}), we get
$$\sum_{k=0}^n  (-1)^k a_k \sum_{j=0}^k C_k^j  \frac{x(x-1)\cdot \ldots \cdot (x-j+1)}{(a)_j} = $$
$$ \sum_{k=0}^n (-1)^k  \frac{a_k}{(a)_k} (x+a)(x+a+1)\cdot \ldots \cdot (x +a+k-1 ) = $$
$$   \sum_{k=0}^n   \frac{a_k}{(a)_k} (-x-a)(-x-a-1)\cdot \ldots \cdot (-x-a -k+1 ) = P(-x-a).$$
Finally, we obtain
\begin{equation}
\label{fd2}
\mathcal{L}_a(P(-x-a)) (x) = Q(-x-1) = \mathcal{L}_a(P)(-x-1).
\end{equation}
Note that if $Q$ has only real zeros lying in the segment $[-1, 0],$ 
then $Q(-x-1)$ also  has only real zeros lying in the segment $[-1, 0].$
Thus, we have proved the following statement.
 
\begin{Statement}
\label{St0}
Let $a>0,$  $P$ be a real polynomial,  
and $\widetilde{P}(x) =P(-x -a).$  If $\mathcal{L}_a(P)$ 
has only real zeros lying in the segment $[-1, 0], $ then $\mathcal{L}_a(\widetilde{P})$ 
 also has only real zeros lying in the segment $[-1, 0]. $
\end{Statement}

\section{Necessary conditions}

Suppose that a polynomial 
$$ P(x)= a_0 + a_1 \frac{x}{a} +a_2 \frac{x(x-1)}{a(a+1)} +  \ldots +  
a_n \frac{x(x-1)\cdot \ldots \cdot (x-n+1)}{a(a+1)\cdot \ldots \cdot (a+n-1)}$$
is such that
$$\mathcal{L}_a(P)
= a_0 +a_1 x +a_2 x^2 + \ldots + a_n x^n$$ 
has only real zeros lying in the segment $[-1, 0].$ We have already noted the necessary 
condition that all coefficients of $\mathcal{L}_a(P)$ have the same sign. 

The following statement gives some further simple necessary conditions.

\begin{Statement}
\label{st1}
Suppose that a real polynomial 
$$P(x)= a_0 + a_1 \frac{x}{a} +a_2 \frac{x(x-1)}{a(a+1)} +  \ldots +  
a_n \frac{x(x-1)\cdot \ldots \cdot (x-n+1)}{a(a+1)\cdot \ldots \cdot (a+n-1)},$$
$a_n>0, $   is such that $\mathcal{L}_a(P)$ 
has only real zeros in the segment $[-1, 0]. $ Then

1. All real zeros of $P$ lie in the segment $[- (n-1 +a), n-1].$

2. For every $j=0, 1, 2, \ldots, n-1,$ we have $\frac{a_j}{a_n} \leq C_n^j.$

3. For every $j= 1, 2, \ldots, n-1$ such that $a_{j-1}>0,$  we have
$$\frac{a_j^2}{a_{j-1}a_{j+1}} \geq \frac{j+1}{j} \cdot \frac{n-j+1}{n-j}$$
(Newton inequalities).

\end{Statement}

\begin{proof}

The statement that all non-negative zeros of $P$ are in the segment $[0, n-1]$
follows from the fact that all coefficients of $P$ have the same sign.
The statement that all non-positive zeros of $P$ are in the segment $[- (n-1 +a), 0]$
follows from the previous statement and formula~(\ref{fd2}). Statements~2 and 3
are well-known necessary conditions for a polynomial to have all real zeros
in the segment  $[-1, 0].$
\end{proof}

It is interesting to compare the following necessary conditions with Theorem~\ref{Th1.0}.

\begin{Theorem}
\label{th1}
Suppose that a real polynomial 
$$P(x)= a_0 + a_1 \frac{x}{a} +a_2 \frac{x(x-1)}{a(a+1)} +  \ldots +  
a_n \frac{x(x-1)\cdot \ldots \cdot (x-n+1)}{a(a+1)\cdot \ldots \cdot (a+n-1)},$$
$a_n>0, $   is such that $Q =\mathcal{L}_a(P)$ 
has only real zeros in the segment $[-1, 0]. $  Denote by $\widetilde{Q}(x)= Q(-x).$   
Then

1. If there exists $k=1, 2, \ldots, n-1$ such that $P(k)=0,$ then
$P(0) = P(1) = \ldots = P(k-1)=0.$

2.  For every $m\in \mathbb{N}$ we have
\begin{equation}
\label{f110}
P(-a- m) = \frac{1}{(a)_m x^{a-1}}  \left(x^{a+ m-1} \widetilde{Q}(x)\right)^{(m)}|_{x=1}.
\end{equation}

3.  If there exists $l \in \mathbb{N}$ such that $P(-a -l)=0,$ then
$P(-a) = P(-a-1) = \ldots = P(-a -(l-1))=0.$

4. If $\deg P = 2l+1, l \in \mathbb{N}\cup \{0\},$  then 
$P$ has a root  in the segment $[-a, 0].$

\end{Theorem}

\begin{proof}

1. We have 
\begin{equation}
\label{f8a}
P(0) = a_0  = Q(0) = \widetilde{Q}(0).
\end{equation} 
For all $k=1, 2, \ldots, n-1,$ we obtain
\begin{equation}
\label{f7}
P(k) = \sum_{j=0}^k \frac{k(k-1)\cdot \ldots \cdot(k-j+1)}{a(a+1) \cdot \ldots \cdot (a+j-1)}  a_j.
\end{equation}
If  $P(k)=0,$ then, since all coefficients $a_j$ have the same sign, we obtain
$a_0=a_1= \ldots = a_k =0$, thus $P(0) = P(1) = \ldots = P(k-1)=0.$

2. For  $P(x)= a_0 + a_1 \frac{x}{a} +a_2 \frac{x(x-1)}{a(a+1)} +  \ldots +  
a_n \frac{x(x-1)\cdot \ldots \cdot (x-n+1)}{a(a+1)\cdot \ldots \cdot (a+n-1)},$ 
we have  $Q(x) = \mathcal{L}_a(P) (x) = a_0 +a_1x +a_2 x^2 + \ldots + a_n x^n,$
and $\widetilde{Q}(x) = Q(-x) =  a_0 - a_1x +a_2 x^2 - \ldots +(-1)^n a_n x^n.$
We obtain
\begin{equation}
\label{f8}
P(- a) = a_0 - a_1 + a_2 -a_3 + \ldots + (-1)^n a_n = Q(-1) =  \widetilde{Q}(1).
\end{equation}
\begin{eqnarray}
\label{f9}
&  P(-a -1) = \frac{a}{a} a_0 -  \frac{a+1}{a} a_1 + \frac{a+2}{a} a_2 - \ldots + (-1)^n \frac{a+n}{a} a_n =
\\ \nonumber & \frac{1}{a x^{a-1}} \left(x^a \widetilde{Q}(x)\right)^\prime|_{x=1}.
\end{eqnarray}
Since all zeros of $Q$ are real and lie in the segment $[-1, 0], $ all zeros of $\widetilde{Q}$ are
real and lie in the segment $[0, 1].$
Thus, all zeros of the function $x^a\widetilde{Q}(x)$ are real and lie in the segment $[0, 1].$ Suppose 
that $P(-a-1)=0.$  This means that  $(x^a \widetilde{Q}(x))^\prime|_{x=1} =0,$  whence  $1$ is a root 
of $\widetilde{Q}$ of multiplicity $\geq 2.$ In particular, $\widetilde{Q}(1) = Q(-1) = P(-a) =0.$

For every $m\in \mathbb{N}$ we get
$$
P(-a- m) =\sum_{k=0}^n (-1)^k a_k \frac{(a+k)(a+k+1) \cdot \ldots 
\cdot (a+k+m-1)}{a(a+1) \cdot \ldots \cdot (a+m-1)} = $$
$$\frac{1}{a(a+1) \cdot \ldots \cdot (a+m-1) x^{a-1}}  \left(x^{a+ m-1} \widetilde{Q}(x)\right)^{(m)}|_{x=1}.
$$

3.  Since  all zeros of $Q$ are real and lie in the segment $[-1, 0],$   all zeros of the function  
$x^{a+ l-1} \widetilde{Q}(x)$ are real and lie in the segment $[0, 1].$ Suppose 
that $P(- a - l)=0.$  This means that  $\left(x^{a+l-1} \widetilde{Q}(x)\right)^{(l)}|_{x=1} =0,$  
whence  $1$ is a root of $\widetilde{Q}$ of multiplicity $\geq l.$ Using (\ref{f110}), we obtain that  
$P(-a) = P(-a-1) = \ldots = P(-a -(l-1))=0.$

4. Let  $\deg P = \deg Q = 2l+1, l \in \mathbb{N}\cup \{0\}. $ Since 
all zeros of $Q$  are real and in the segment $[-1, 0],$  we have $Q(0)\cdot Q(-1) \leq 0.$
By (\ref{f8a}) and (\ref{f8}) we get $P(0)\cdot P(-a) \leq 0.$ Therefore, $P$ has a root in 
the  segment $[-a, 0].$
\end{proof}

In fact, we have proved the following statement concerning the possible zeros of a polynomial 
$Q$ at the endpoints of a segment $[-1, 0].$

\begin{Statement}
\label{st2}
Suppose that a real polynomial 
$$P(x)= a_0 + a_1 \frac{x}{a} +a_2 \frac{x(x-1)}{a(a+1)} +  \ldots +  
a_n \frac{x(x-1)\cdot \ldots \cdot (x-n+1)}{a(a+1)\cdot \ldots \cdot (a+n-1)},$$
$a_n>0, $   is such that $Q =\mathcal{L}_a(P)$ 
has only real zeros in the segment $[-1, 0]. $ 

1. The following four statements are equivalent:

1a. The polynomial $Q$ has a zero of
multiplicity $m\in \mathbb{N}$ at the point $x=0.$ 

1b. $a_0 = a_1 = \ldots =a_{m-1}=0.$

1c. $P(m-1)=0.$

1d. $P(0) = P(1) = \ldots = P(m-1)=0.$

2. The following three statements are equivalent.

2a. The polynomial $Q$ has a zero of
multiplicity $m\in \mathbb{N}$ at the point $x=-1.$

2b. $P(-(a + m-1))=0.$

1c. $P(-(a + m-1)) = P(-(a + m-2)) = \ldots = P(-a)=0.$

\end{Statement}

\section{Proof of Theorem~\ref{Th1.0}}

To prove Theorem~\ref{Th1.0}, we need the following lemma.

\begin{Lemma}
\label{L1.1.1}
Let $a>0$ and $P_n\in \mathbb{R}[x]$ be a real polynomial whose zeros are all real, 
with  $\deg P_n =n,   n\in \mathbb{N}.$ Let
\begin{equation}
\label{f9.1}
\sum_{k=0}^\infty \frac{(a)_k}{k!} P_n(k) x^k = \frac{S_m(x)}{(1-x)^{n+a}},
\end{equation}
where $S_m\in \mathbb{R}[x]$ is a real polynomial, $\deg S_m =m \leq n, m\in 
\mathbb{N}\cup \{0\},$ and all the zeros of $S_m$ are real and non-positive.
Suppose that $ r\in \mathbb{N}\cup \{0\}$ is such an integer that all the zeros of
$P_n$ lie in the interval $(- \infty, r]$ and $P_n(0)=P_n(1) = \ldots = P_n(r)=0.$
Then for every $\beta$ such that $ r < \beta \leq r+1,$ we have
\begin{equation}
\label{f9.2}
\sum_{k=0}^\infty \frac{(a)_k}{k!} P_n(k) (k-\beta) x^k = \frac{S_{m+1}(x)}{(1-x)^{n+1+a}},
\end{equation}
where $S_{m+1}\in \mathbb{R}[x]$ is a real polynomial, $\deg S_m =m+1, $ and all the 
zeros of $S_{m+1}$ are also real and non-positive.

\end{Lemma}

\begin{proof}
Using (\ref{f9.1}), we obtain 
$$\sum_{k=0}^\infty \frac{(a)_k}{k!} P_n(k) (k-\beta) x^k = x \left( \frac{S_m(x)}{(1-x)^{n+a}}\right)^\prime (x) 
-\beta  \frac{S_m(x)}{(1-x)^{n+a}} =$$
$$  \frac{x(1-x) S_m^\prime (x) + ((n+a+\beta)x - \beta) S_m(x)}{(1-x)^{n+1+a}},  $$
whence
\begin{equation}
\label{f9.3}
S_{m+1} (x) = x(1-x) S_m^\prime (x) + ((n+a+\beta)x - \beta) S_m(x).
\end{equation}
By our assumptions, $P_n(0)=P_n(1) = \ldots = P_n(r)=0, $ and  $P_n (r+1) \ne 0.$  Using
Statement~\ref{st2}, we obtain for  $Q_m =\mathcal{L}_a(P_n)$ that
$x^{r+1} \mid Q_m$ and $x^{r+2} \nmid Q_m.$ By (\ref{ff1}), $S_{m}(x) =(1-x)^n Q_{m}\left(\frac{x}{1-x}\right),$
hence $x^{r+1} \mid S_m,\   x^{r+2} \nmid S_m,$ which means that
$$ S_m(x) = x^{r+1} (x+x_1)(x+x_2) \cdot \ldots \cdot (x+x_{m-r-1}), x_j > 0, 
1\leq j \leq m-r-1.   $$
Without loss of generality, we assume that $0< x_1 < x_2 < \ldots < x_{m-r-1}.$

If $P_{n+1} = P_n(x) (x- \beta),$ then $P_{n+1}(0)=P_{n+1}(1) = \ldots = P_{n+1}(r)=0, $
whence $x^{r+1} \mid S_{m+1}.$

We write
$$ S_m(x) = b_m x^m +b_{m-1} x^{m-1} + \ldots + b_{r+1}x^{r+1}, $$
where, without loss of generality, $b_m >0.$ Since all the zeros of $S_m$ 
are real and non-positive, all coefficients of $S_m$ have the same sign, so
$ b_{r+1}>0.$ So by (\ref{f9.3}) we have
\begin{eqnarray}
\label{f9.4}
&S_{m+1}(x)= x(1-x)(m b_m x^{m-1} + \ldots + (r+1) b_{r+1} x^r)  +
\\ \nonumber &  ((n+a+\beta)x - \beta) (b_m x^m + \ldots + b_{r+1}x^{r+1})  =   
\\  \nonumber &   b_m(n+a+ \beta -m) x^{m+1}  + \ldots + b_{r+1} (r+1 - \beta)x^{r+1}.  
\end{eqnarray}
We note that $n+a+ \beta -m >0$ and $ r+1 - \beta >0.$

1. Suppose that $m\in 2\mathbb{N} \cup \{0\}.$ Then $\sgn S_m (-\infty) = +1,\     
\sgn S_{m+1} (-\infty) = -1,\    $ and
$$\sgn S_m^\prime (-x_{m-r-1}) = -1,    \sgn S_m^\prime (-x_{m-r-2}) = +1, 
\ldots , $$
$$\sgn S_m^\prime (-x_{1}) = (-1)^{m-r-1}= (-1)^{r+1}.$$
From (\ref{f9.3}) we obtain
$$\sgn S_{m+1} (-x_{m-r-1})  =\sgn \left(-x_{m-r-1} (1+x_{m-r-1})S_m^\prime (-x_{m-r-1}) \right) = +1; $$
$$\sgn S_{m+1} (-x_{m-r-2})  =\sgn \left(-x_{m-r-2} (1+x_{m-r-2})S_m^\prime (-x_{m-r-2}) \right) = -1; $$
$$\sgn S_{m+1} (-x_{m-r-3})  =\sgn \left(-x_{m-r-3} (1+x_{m-r-3})S_m^\prime (-x_{m-r-3}) \right) = +1; $$
$$\vdots $$
$$ \sgn S_{m+1} (-x_{1})  =\sgn \left(-x_{1} (1+x_{1})S_m^\prime (-x_{1}) \right) = (-1)^r. $$
Moreover, by (\ref{f9.4})  we have $ \sgn S_{m+1} (-x)  =(-1)^{r+1}$ for $x$ being   negative 
and sufficiently close to $ 0.$

We have proved that on each of the intervals 
$$(-\infty, -x_{m-r-1}), (-x_{m-r-1}, -x_{m-r-2}), \ldots,   (-x_2, -x_1), (-x_1, 0)$$   
there is a root of the polynomial $S_{m+1}.$ Hence, $S_{m+1}$ has at least $m-r$ negative roots 
and a zero of multiplicity at least $r+1$ at the point $0.$ Thus, all $m+1$ roots of the polynomial $S_{m+1}$ are real
and non-positive.

2. Suppose now that $m -1 \in 2\mathbb{N} \cup \{0\}.$ Then $\sgn S_m (-\infty) = -1,  
\sgn S_{m+1} (-\infty) = +1, $ and
$$\sgn S_m^\prime (-x_{m-r-1}) = +1,    \sgn S_m^\prime (-x_{m-r-2}) = -1, 
\ldots , $$
$$\sgn S_m^\prime (-x_{1}) = (-1)^{m-r }= (-1)^{r-1}.$$
From (\ref{f9.3}) we obtain
$$\sgn S_{m+1} (-x_{m-r-1})  =\sgn \left(-x_{m-r-1} (1+x_{m-r-1})S_m^\prime (-x_{m-r-1}) \right) = -1; $$
$$\sgn S_{m+1} (-x_{m-r-2})  =\sgn \left(-x_{m-r-2} (1+x_{m-r-2})S_m^\prime (-x_{m-r-2}) \right) = +1; $$
$$\sgn S_{m+1} (-x_{m-r-3})  =\sgn \left(-x_{m-r-3} (1+x_{m-r-3})S_m^\prime (-x_{m-r-3}) \right) = -1; $$
$$\vdots $$
$$ \sgn S_{m+1} (-x_{1})  =\sgn \left(-x_{1} (1+x_{1})S_m^\prime (-x_{1}) \right) = (-1)^r. $$
Besides that, by (\ref{f9.4})  we have $ \sgn S_{m+1} (-x)  =(-1)^{r+1}$ for $x$ being   negative 
and sufficiently close to $ 0.$

We have proved that on each of the intervals 
$$(-\infty, -x_{m-r-1}), (-x_{m-r-1}, -x_{m-r-2}), \ldots,   (-x_2, -x_1), (-x_1, 0)$$   
there is a root of the polynomial $S_{m+1}.$ Hence, $S_{m+1}$ has at least $m-r$ 
negative roots and  a zero of multiplicity at least $r+1$ at the point $0.$ Thus, all $m+1$ 
roots of the polynomial $S_{m+1}$ are real
and non-positive.

Lemma~\ref{L1.1.1}  is proved. 
\end{proof}

Using the previous lemma, we can easily prove Theorem~\ref{Th1.0}. Let
$a>0,$ and $P\in \mathbb{R}[x]$ be a polynomial with only real zeros. Denote by 
$\lambda(P), \Lambda(P)$ the smallest and the largest zeros of $P,$  respectively. Suppose 
that $P(x)=0$ for all $x\in \left(\left[\lambda(P), -a\right] \cap \{ -a, -a-1, -a-2, \ldots  \}\right)
\cup \left(\left[0, \Lambda(P)\right] \cap \mathbb{Z}\right).$  Let $x_1 \leq x_2 \leq \ldots 
\leq x_n$ be the zeros of $P.$

If $x_n \leq 0,$ then, by Theorem~C, the polynomial $\mathcal{L}_a(P)$ 
has only real zeros lying in the segment $[-1, 0].$ Now suppose that 
$x_1 \leq x_2 \leq \ldots \leq x_l =0 < x_{l+1} \leq x_{l+2} \leq \ldots 
\leq x_n.$ Let $P_l (x) =(x-x_1)(x-x_2) \cdot \ldots \cdot (x-x_l).$
By Theorem~C, we conclude that the polynomial  $Q_l :=\mathcal{L}_a(P_l)$ 
has only real zeros lying in the segment $[-1, 0].$ This is equivalent to the
statement that
$$\sum_{k=0}^\infty \frac{(a)_k}{k!} P_l(k) x^k = \frac{S_t(x)}{(1-x)^{a+l}}, $$
where $S_t$ is a real polynomial, $\deg S_t = t \leq l,$ and all zeros of  $S_t$
are real and non-positive. Under the assumptions of Theorem~\ref{Th1.0}, after 
recursively applying  Lemma~\ref{L1.1.1} with $\beta =  x_{l+1},  x_{l+2},  \ldots,  x_n, $ 
we obtain
$$\sum_{k=0}^\infty \frac{(a)_k}{k!} P(k) x^k = \frac{S_q(x)}{(1-x)^{a+n}}, $$
where $S_q$ is a real polynomial, $\deg S_q = q \leq n=\deg P,$ and all zeros of  $S_q$
are real and non-positive. This is equivalent to the fact that the polynomial $\mathcal{L}_a(P)$ 
has only real zeros lying in the segment $[-1, 0].$ 

Theorem~\ref{Th1.0} is proved. $\Box$

\section{Sufficient conditions}

Let  $Q(x) = \sum_{k=0}^n a_k x^k$  be a polynomial with positive coefficients. 
We define the second quotients  $q_k$ as follows:

\begin{equation}
\label{qqq} 
q_k=q_k(Q):=\frac {a_{k-1}^2}{a_{k-2}a_k},\
2 \leq k \leq n.
\end{equation}

The formulas below follow from repeated application of the definition of the second quotients.

\begin{equation}
\label{defq}
 a_k=\frac
{a_1}{q_2^{k-1} q_3^{k-2} \ldots q_{k-1}^2 q_k} \left(
\frac{a_1}{a_0} \right) ^{k-1},\    2 \leq k \leq n.
\end{equation}

The second quotients $q_k$ provide useful information about the distribution of zeros 
of the associated polynomial. This connection dates back to the classical work of 
J.I. Hutchinson~\cite{hut}. In 1926, Hutchinson established one of the first general 
coefficient-based criteria for a polynomial (or an entire function)  with positive 
coefficients  to have only real non-positive zeros.  

{\bf Theorem E} (J. I. ~Hutchinson, \cite{hut}). { \it Let $Q(x)=
\sum_{k=0}^n a_k x^k$, $a_k > 0$ for all $k$. 
If
$$q_k(Q)\geq 4 \ \ \mbox{ for all} \ \ 2 \leq k \leq n,$$  
then all zeros of $Q$ are real and negative.}

It is easy to show that, if only a lower bound for $q_k(Q)$ is given, then the constant   
$4$  in $q_k (Q) \geq 4$ is the smallest possible constant for concluding that
$Q$ has only real zeros. However, if we have both a lower and an upper bound for $q_k(Q),$  
then the constant $4$ in the condition $q_k(Q) \geq 4$ can be reduced to 
conclude that all zeros of a polynomial  $Q$  are real and negative, see  \cite{HishAn}.

{\bf Theorem F} (T.H.~ Nguyen, A.~Vishnyakova,  \cite{HishAn}).
{ \it  Let $Q(x) = \sum_{k=0}^n a_k x^k,$ $a_k > 0,$  be a polynomial, and $n \geq 4.$ Suppose that 
there exists $\alpha,  1 + \sqrt{5} \leq \alpha < 4,$ such that
 $q_k(Q) \in \left[\alpha, \frac{8}{\alpha(4 - \alpha)} \right]$ for  $k =2, 3, \ldots, n.$  
Then  all zeros of $Q$ are real and negative. }

The following theorem is a corollary of Theorems~E and~F.

\begin{theorem}
\label{th2}
Let $P$ be the real polynomial 
$$P(x)= a_0 + a_1 \frac{x}{a} +a_2 \frac{x(x-1)}{a(a+1)} +  \ldots +  
a_n \frac{x(x-1)\cdot \ldots \cdot (x-n+1)}{a(a+1)\cdot \ldots \cdot (a+n-1)},$$
$a_j>0 $ for all $j=0, 1, \ldots, n.$

1. Suppose that $P$ satisfies the conditions
$$ \frac {a_{k-1}^2}{a_{k-2}a_k} \geq 4, \    2 \leq k \leq n, $$
and
$$ \frac{a_{n-1}}{a_n} \leq 1. $$
Then $\mathcal{L}_a(P)$  has only real zeros in the segment $[-1, 0]. $

2. Suppose that $n \geq 4$ and there exists $\alpha,  1 + \sqrt{5} \leq \alpha < 4,$ such that
 $\frac {a_{k-1}^2}{a_{k-2}a_k} \in \left[\alpha, \frac{8}{\alpha(4 - \alpha)} \right]$ for all 
 $k =2, 3, \ldots, n,$  and $ \frac{a_{n-1}}{a_n} \leq 1. $  Then $\mathcal{L}_a(P)$ 
has only real zeros in the segment $[-1, 0]. $

\end{theorem}

\begin{proof}
Theorems E and F show that, under the assumptions of Theorem~\ref{th2}, 
all zeros of  $Q(x) = \mathcal{L}_a(P) (x) = a_0 +a_1x +a_2 x^2 + \ldots + 
a_n x^n$ are real and negative.  It remains to show that all these zeros lie in the 
segment  $[-1, 0].$ Under the assumptions of Theorem~\ref{th2} 
we have 
$$0 < \frac{a_0}{a_1} < \frac{a_1}{a_2} <  \ldots < \frac{a_{n-2}}{a_{n-1}} < \frac{a_{n-1}}{a_n}.  $$
Therefore, for every $x\in  (-\infty, - \frac{a_{n-1}}{a_n}]$ we have
$$ a_0 < a_1 |x| < a_2 |x|^2 < \ldots < a_{n-1}|x|^{n-1} < a_n |x|^n.  $$
So, for all $x\in  (-\infty, - \frac{a_{n-1}}{a_n}]$ we obtain
$$(-1)^n Q(x) = (a_n |x|^n - a_{n-1}|x|^{n-1}) + (a_n |x|^{n-2} - a_{n-1}|x|^{n-3}) +\ldots >0.  $$
Hence, $Q$ has no zeros  in $ (-\infty, - \frac{a_{n-1}}{a_n}],$ and since $ \frac{a_{n-1}}{a_n} \leq 1$
all zeros of $Q$ are in the segment $[-1, 0].$ 
\end{proof}

\section{Some examples}

We  start with the following trivial example.

\begin{Example}
\label{ex1} For a polynomial $P$ of degree $1,$ we have $\mathcal{L}_a(P)(x) = P(ax),$
so the polynomial $\mathcal{L}_a(P)$ has  one real zero lying in the segment $[-1, 0]$ if
and only if the polynomial $P$ has  one real zero lying in the segment $[-a, 0].$
\end{Example}

The following examples demonstrate that sufficient conditions in Theorem~\ref{Th1.0}
are not necessary.
\begin{Example}
\label{ex2.1} 
  Let 
$$P(x)= \left(x+ \frac{a+\sqrt{a(a+1)} }{2}\right)^2 = a(a+1) \frac{x(x-1)}{a(a+1)}+$$
$$ a(a+1 +\sqrt{a(a+1)})\frac{x}{a} +  
 \left(\frac{a+\sqrt{a(a+1)} }{2}\right)^2, $$
then
$$ Q(x)= \mathcal{L}_a(P)(x) = a(a+1)x^2 +  a(a+1 +\sqrt{a(a+1)})x + $$
$$\left(\frac{a+\sqrt{a(a+1)} }{2}\right)^2= \left(\sqrt{a(a+1)}x +  \frac{a+\sqrt{a(a+1)} }{2}  \right)^2. $$
The polynomial $ Q$  has  only real zeros located in the 
segment $[-1, 0].$   We see that the polynomial  $P$ has  real negative zeros with modulus greater than $a$, 
but $P(-a) \ne 0.$ 

Let
$$P(x)= \left(x+ \frac{a-\sqrt{a(a+1)} }{2}\right)^2 = a(a+1) \frac{x(x-1)}{a(a+1)}+$$
$$ a(a+1 -\sqrt{a(a+1)})\frac{x}{a} +  
 \left(\frac{a-\sqrt{a(a+1)} }{2}\right)^2, $$

then
$$ Q(x)= \mathcal{L}_a(P)(x) = a(a+1)x^2 +  a(a+1 -\sqrt{a(a+1)})x + $$
$$\left(\frac{a -\sqrt{a(a+1)} }{2}\right)^2= \left(\sqrt{a(a+1)}x -  \frac{a -\sqrt{a(a+1)} }{2}  \right)^2. $$
The polynomial $ Q$  has  only real  zeros located in the segment $[-1, 0].$
We see that the polynomial  $P$ has positive zeros,  but $P(0) \ne 0.$ 
\end{Example}

\begin{Example}
\label{ex2a}  
Let $P(x)= x^2 + \frac{a}{4(a+1)} = a(a+1)\cdot \frac{x(x-1)}{a(a+1)}+ a \frac{x}{a} + \frac{a}{4(a+1)},$
then
$$ Q(x)= \mathcal{L}_a(P)(x) = a(a+1)x^2 + ax + \frac{a}{4(a+1)}=\frac{a}{4(a+1)}\left(2(a+1)x+1\right)^2. $$
The polynomial $ Q$  has only real zeros lying in the segment $[-1, 0].$
But the polynomial $P$ has all non-real zeros.  
\end{Example}

\begin{Example}
\label{ex2d} 
Let $P_{n, t}(x) =\left(x + t\right)^n =:
\sum_{k=0}^n a_k(t)   \frac{x(x-1)\cdot \ldots \cdot (x-k+1)}{(a)_k}, 
n\in \mathbb{N}.$ Using (\ref{fd1}), we obtain
$$ a_k(t)= \frac{(a)_k}{k!}  \sum_{j=0}^k (-1)^j   C_k^j P_{n, t}(k-j) =\frac{(a)_k}{k!}  \sum_{j=0}^k (-1)^j   C_k^j 
(k-j+t)^n,  0\leq k \leq n.
 $$
 Then
 $$ Q_{n, t}(x)= \mathcal{L}_a(P_{n, t})(x) = \sum_{k=0}^n \frac{(a)_k}{k!} x^k  \sum_{j=0}^k (-1)^j   C_k^j 
(k-j+t)^n.$$

For a given $n\in \mathbb{N}$ we want to find the set of values of $t$ such that 
all zeros of $ Q_{n, t}$ are real and lie in the segment $[-1, 0].$ We denote this set by 
$$M_n(a) = \{ t\in \mathbb{R} | \  \mbox{all zeros of}\    \mathcal{L}_a(P_{n, t})  \   
\mbox{are real and lie  in the segment}\   [-1, 0]  \}.$$
By (\ref{fd2}), if $\mathcal{L}_a(P_{n,t})$ 
has only real zeros in the segment $[-1, 0], $ then $\mathcal{L}_a(\widetilde{P}_{n,t})$ 
 has only real zeros in the segment $[-1, 0],  $ where $\widetilde{P}_{n,t}(x) =P_{n,t}(-x -a).$
Thus, $t\in M_n(a) \Rightarrow (a-t)\in M_n(a),$
so it is sufficient to consider the case $t \geq \frac{a}{2}.$

Further we will assume that $t\geq0.$

The answer is simple for odd  values of $n.$  If $n$  is odd, $Q_{n, t}(0)=P_{n, t}(0)= t^n >0$  and 
all zeros of $ Q_{n, t}$ are real and lie in the segment $[-1, 0], $ then  $Q_{n, t}(-1)\leq 0.$
It is easy to check that  $Q_{n, t}(-1)  = P_{n, t}(-a) =(-a+t)^n \leq 0$  (see (\ref{f8})). So, $t\leq a.$
Hence, $t\in M_n(a),  t\geq 0 \Rightarrow t \in [0, a].$ On the other hand, using Theorem~\ref{Th1.0}
we get $[0, a] \subset M_n(a).$ Whence, for every $k\in \mathbb{N}\cup\{0\}$ we have $M_{2k+1}(a) = [0, a].$

For even $n,$ the situation is much more complicated. We note that, by statement 3 in 
Theorem~\ref{th1},  that for all $k, j\in \mathbb{N}, j \geq 2,$
we have  $-a-j \notin M_{2k}(a).$ Numerical calculations suggest that the following 
conjecture holds.

\begin{Conjecture} For  every $k\in \mathbb{N},$ we have $M_{2k} (a)= [-c_{2k}(a), a+ c_{2k}(a)], 0 < c_{2k}(a) <2a,$ 
and $c_2(a) \leq c_4(a) \leq c_6(a) \leq \ldots$ 
\end{Conjecture}

The proof of this fact and the possible value of the limit $\lim_{k\to \infty} c_{2k}(a) $ remain open. 

\end{Example} 

\begin{Example}
\label{ex3}
 For $n \in \mathbb{N},$ consider a polynomial 
\begin{equation}
\label{f5}
P_{0,n}(x) = \frac{x(x-1)\cdot \ldots \cdot (x-n+1)}{(a)_n}.
\end{equation}
We have $\mathcal{L}_a(P_{0,n}) = x^n,$ so the polynomial $\mathcal{L}_a(P_{0,n})$ has only real zeros lying
in the segment $[-1, 0].$ The zeros of $P_{0,n}$ are $\{0, 1, 2,  \ldots, n-1\},$ so these zeros satisfy
the sufficient conditions of Theorem~\ref{Th1.0}.

For a given $n \in \mathbb{N}$ and $s=0, 1, 2, \ldots, n,$ we consider the polynomial
\begin{eqnarray}
\label{f66}
&
P_{s,n}(x) =
\\ \nonumber &
 \frac{(x+a)(x+a+1)\cdot \ldots \cdot (x+a+s-1)\cdot x (x-1) \cdot 
\ldots \cdot (x-(n-s-1))}{(a)_n} .
\end{eqnarray}

Let $Q_{s,n}(x) = \mathcal{L}_a (P_{s,n})(x). $   Since $P_{s,n}(0)=P_{s,n}(1)= \ldots = P_{s,n}(n-s-1) =0,$ by 
Statement~\ref{st2}, we obtain $x^{n-s} \mid Q_{s,n}.$ Since $P_{s,n}(-a)=P_{s,n}(-a -1)= \ldots = P_{s,n}(-a-s +1) =0,$ by 
Statement~\ref{st2}, we obtain $(x+1)^s \mid Q_{s,n}.$ We have $\deg Q_{s,n} =n,$ and thus $Q_{s,n}(x)= C 
x^{n-s} (x+1)^s, C \in \mathbb{R}.$ The leading coefficient of  $P_{s,n}$ is equal to $\frac{1}{(a)_n},$
whence $C=1.$ Finally, we get

\begin{equation}
\label{f511}
Q_{s,n}(x) = \mathcal{L}_a (P_{s,n})(x) = x^{n-s} (x+1)^s.
\end{equation}

For $s=n,$ we have $P_{n,n}(x) =  \frac{(x+a)(x+a+1)\cdot \ldots \cdot (x+a+n-1)}{(a)_n} $ and   
$\mathcal{L}_a (P_{n,n})(x) = (1+x)^n $
(cf.  (\ref{f6})).

The zeros of $P_{s,n}$ are $\{ -(a+s-1), -(a+s-2), \ldots, -(a+1), -a, 0, 1,  \ldots, n-s-1\},$ so these zeros satisfy
the  sufficient conditions of  Theorem~\ref{Th1.0}.

\end{Example}

\begin{Example}
\label{ex4}
For a given $n \in \mathbb{N}\cup\{0\},$ consider a polynomial 
\begin{equation}
\label{f68}
Q_{n, a}(x)= (a)_n \left(x+ \frac{1}{2}\right)^n =  (a)_n \sum_{k=0}^n C_n^k \frac{x^k}{2^{n-k}}.
\end{equation}
Formula (\ref{fd2}) shows that if the set of zeros of a polynomial $Q_{n, a} = \mathcal{L}_a(P_{n, a})$ is symmetric 
with  respect to the point $\frac{1}{2},$ then  the set of zeros of the polynomial $P_{n, a}$ is  symmetric with
respect to the point $\frac{a}{2}.$  The family of polynomials  $Q_{n, a}(x)= (a)_n \left(x+ \frac{1}{2}\right)^n$ 
is of particular interest because, for every $n,$ the polynomial $Q_{n, a}$ has zeros that are, in a natural sense, 
maximally symmetric  with  respect to  the point $\frac{1}{2}.$

We have
\begin{equation}
\label{f69}
P_{n, a}(x) = \mathcal{L}_a^{-1} \left(Q_{n, a}\right)(x) = (a)_n \sum_{k=0}^n C_n^k \frac{x(x-1)(x-2)
\cdot \ldots \cdot (x-k+1)}{(a)_k 2^{n-k}}.
\end{equation}
Note that $Q_{n, a}(-1-x) = (-1)^n Q_{n, a}(x). $ Thus, by (\ref{fd2})
\begin{equation}
\label{f70}
P_{n, a}(-a-x) = (-1)^n P_{n, a}(x).
\end{equation}
Define
\begin{equation}
\label{f71}
S_{n, a}(x) = P_{n, a}\left(- \frac{a}{2} +x\right),
\end{equation}
so that
\begin{equation}
\label{f72}
S_{n, a}(- x) =  (-1)^n S_{n, a}(x).
\end{equation}
Direct computation gives
$$S_{0, a}(x)=1,$$
$$S_{1, a}(x) =x,$$
$$S_{2, a}(x)= x^2 +\frac{a}{4},$$
$$S_{3, a}(x)= x^3  +\left(\frac{3a}{4}+ \frac{1}{2}\right)x,$$
$$S_{4, a}(x)= x^4 +\left(\frac{3}{2}a+2\right) x^2 + \left(\frac{3a^2}{16}  +\frac{3a}{8}\right).$$
It is easy to check that the zeros of the polynomials $S_{n, a}, 0\leq k \leq 4,$
are purely imaginary, so all zeros of polynomials $P_{n, a}, 0\leq n \leq 4,$
have real parts equal $- \frac{a}{2}.$ For the case $a=1$ this fact was proved in \cite{vish}.

\begin{theorem}
\label{th1.0} For every $n\in \mathbb{N}\cup\{0\},$ all zeros of the polynomial
$P_{n, a}(x) = (a)_n \sum_{k=0}^n C_n^k \frac{x(x-1)(x-2)
\cdot \ldots \cdot (x-k+1)}{(a)_k 2^{n-k}}$ are simple and have real parts equal to
$- \frac{a}{2}.$ Equivalently, all zeros of the polynomial
$S_{n, a}(x) = P_{n, a}\left(- \frac{a}{2} +x\right)$ are simple and purely imaginary.
Moreover,  the zeros of polynomials $P_{n, a}$ and $P_{n+1, a}$ interlace for all
$n\in \mathbb{N}.$ Equivalently, the zeros of polynomials $S_{n, a}$ and 
$S_{n+1,a}$ interlace for all  $n\in \mathbb{N}.$
\end{theorem}

\begin{proof}
By (\ref{f69}) and  (\ref{f71}) we have
\begin{equation}
\label{f73}
S_{n, a}(x) =  (a)_n \sum_{k=0}^n \frac{C_n^k }{(a)_k 2^{n-k}} \prod_{j=1}^k \left(x- \frac{a}{2} -j +1\right).
\end{equation}

We first prove the following lemma.
\begin{Lemma}
\label{L1}
For all $n \geq 2,$  the following recurrence relation holds
\begin{equation}
\label{f74}
S_{n, a}(x) =  x S_{n-1, a}(x) + \frac{(n-1)(a+n-2)}{4} S_{n-2, a}(x).
\end{equation}
\end{Lemma}

\begin{proof}
We have
$$ x S_{n-1, a}(x) +  \frac{(n-1)(a+n-2)}{4}  S_{n-2, a}(x) = $$  
$$x (a)_{n-1} \sum_{k=0}^{n-1} \frac{C_{n-1}^k }{(a)_k 2^{n-1-k}}
 \prod_{j=1}^k \left(x- \frac{a}{2} -j +1\right)  +$$
$$ \frac{(n-1)(a+n-2)}{4} (a)_{n-2}  \sum_{k=0}^{n-2} \frac{C_{n-2}^k }{(a)_k 2^{n-2-k}}
 \prod_{j=1}^k  \left(x- \frac{a}{2} -j +1\right)   = $$
$$=    (a)_{n-1} \sum_{k=0}^{n-1} \frac{C_{n-1}^k }{(a)_k 2^{n-1-k}}
 \prod_{j=1}^{k} \left(x-  \frac{a}{2} -j+1\right)\cdot \left(x - \frac{a}{2} - k\right)  +$$   
$$  (a)_{n-1}  \sum_{k=0}^{n-1} \frac{C_{n-1}^k }{(a)_k 2^{n-1-k}} 
 \prod_{j=1}^k \left(x-  \frac{a}{2} -j+1\right) \cdot \left( \frac{a}{2} + k\right) +$$
 $$  \frac{(n-1)(a+n-2)}{4} (a)_{n-2}  \sum_{k=0}^{n-2} \frac{C_{n-2}^k }{(a)_k 2^{n-2-k}}
 \prod_{j=1}^k \left(x- \frac{a}{2} -j +1\right)  =  $$
 $$    (a)_{n-1}\sum_{k=0}^{n-1} \frac{C_{n-1}^k }{(a)_k 2^{n-1-k}}
 \prod_{j=1}^{k+1} \left(x-  \frac{a}{2} -j+1\right)  $$   
$$+(a)_{n-1} \sum_{k=0}^{n-1} \frac{C_{n-1}^k }{(a)_k 2^{n-1-k}} \cdot \left( \frac{a}{2} + k\right) 
 \prod_{j=1}^k \left(x- \frac{a}{2} -j +1\right) + $$
 $$ \frac{(n-1)}{4} (a)_{n-1}   \sum_{k=0}^{n-2} \frac{C_{n-2}^k }{(a)_k 2^{n-2-k}}
 \prod_{j=1}^k \left(x- \frac{a}{2} -j +1\right) = $$
$$  \left(\frac{a \cdot (a)_{n-1}}{2^n} +\frac{(a)_{n-1}\cdot(n-1)}{2^n}  \right)  +$$
$$  \sum_{k=1}^{n-2} 
\prod_{j=1}^k \left(x- \frac{a}{2} -j +1\right) \cdot \left(\frac{(a)_{n-1}C_{n-1}^{k-1}}{2^{n-k} (a)_{k-1}} + 
\frac{(a)_{n-1} C_{n-1}^{k}}{2^{n-k-1} (a)_k}\cdot   \left( \frac{a}{2} + k\right)  \right.  +$$
$$\left.  \frac{(n-1)}{4} \cdot 
\frac{(a)_{n-1} C_{n-2}^k}{2^{n-k-2}(a)_k} \right)  + $$
$$\prod_{j=1}^{n-1} \left(x- \frac{a}{2} -j +1\right)\cdot 
\left(\frac{(a)_{n-1} C_{n-1}^{n-2}}{2^1 (a)_{n-2}} +
\frac{(a)_{n-1} C_{n-1}^{n-1}}{2^0 (a)_{n-1}}\cdot (\frac{a}{2}+n-1) \right) + $$
$$\prod_{j=1}^n \left(x- \frac{a}{2} -j +1\right)\cdot 
\frac{(a)_{n-1} C_{n-1}^{n-1}}{2^0 (a)_{n-1}} $$
$$=\frac{(a)_n}{2^n}+ \sum_{k=1}^{n-2}  b_k
\left(x- \frac{a}{2} -j +1\right) + \prod_{j=1}^{n-1} \left(x- \frac{a}{2} -j +1\right)\cdot \frac{n(a+n-1)}{2} +$$
$$  \prod_{j=1}^n \left(x- \frac{a}{2} -j +1\right),$$
where
$$b_k =  \left(\frac{(a)_{n-1}C_{n-1}^{k-1}}{2^{n-k} (a)_{k-1}} + 
\frac{(a)_{n-1} C_{n-1}^{k}}{2^{n-k-1} (a)_k}\cdot   \left( \frac{a}{2} + k\right)  + \frac{(n-1)}{4} \cdot 
\frac{(a)_{n-1} C_{n-2}^k}{2^{n-k-2}(a)_k} \right)     $$
$$= \frac{(a)_{n-1}}{(a)_k 2^{n-k}}\left((a+k-1) C_{n-1}^{k-1} +2C_{n-1}^{k}\cdot   \left( \frac{a}{2} + k\right)  +
(n-1) C_{n-2}^{k} \right)       $$
$$ = \frac{(a)_{n-1}}{(a)_k 2^{n-k}}\left(\frac{(a+k-1)(n-1)!}{(k-1)!(n-k)!} +\frac{(2k+a)(n-1)!}{k!(n-1-k)!} 
+ \frac{(n-1)!}{k!(n-2-k)!}   \right)    $$
$$    =  \frac{(a)_{n-1}}{(a)_k 2^{n-k}} \cdot \frac{(n-1)!}{k!(n-k)!} \cdot \left(k(a+k-1) +(2k+a)(n-k) +(n-1-k)(n-k)\right)$$
$$ =  \frac{(a)_{n-1}}{(a)_k 2^{n-k}} \cdot \frac{(n-1)!}{k!(n-k)!} \cdot n(a+n-1)  =(a)_n \cdot \frac{C_n^k}{2^{n-k}(a)_k}.   $$
Finally, we obtain
$$x S_{n-1, a}(x) +  \frac{(n-1)(a+n-2)}{4}  S_{n-2, a}(x) = $$
$$ \frac{(a)_n}{2^n}+ \sum_{k=1}^{n-2}  (a)_n \cdot \frac{C_n^k}{2^{n-k}(a)_k}
\prod_{j=1}^k \left(x- \frac{a}{2} -j +1\right) + \prod_{j=1}^{n-1} \left(x- \frac{a}{2} -j +1\right)\cdot \frac{n(a+n-1)}{2} +$$
$$  \prod_{j=1}^n \left(x- \frac{a}{2} -j +1\right)  = S_{n, a}(x). $$
Lemma~\ref{L1} is proved.
\end{proof}

We now make a change of variables to pass from the imaginary axis to the real axis. 
We define a new sequence of polynomials
$$M_{n, a}(x) = i^{-n} S_{n, a}(ix).$$
Using (\ref{f72}), we get that $M_{n, a}$ is a real polynomial. From (\ref{f74}), we obtain
$$S_{n, a}(ix) =  ix S_{n-1, a}(ix) +\frac{(n-1)(a+n-2)}{4}  S_{n-2, a}(ix),   $$
or
\begin{equation}
\label{f76}
M_{n, a}(x) = x M_{n-1, a}(x) -\frac{(n-1)(a+n-2)}{4} M_{n-2, a}(x).
\end{equation}
We also know that
$$M_{0, a}(x)=1,   M_{1, a}(x) =x.$$
We now use the classical theorem of Favard.

{\bf Theorem G} (J.~Favard \cite[pp. 15–16]{RaSc}).  {\it     
Let $(K_n(x))_{n=0}^\infty \subset \mathbb{R}[x],  \deg K_n =n,$ be a sequence of
real polynomials satisfying a recurrence relation of the form
$$K_n(x)=(x-a_n) K_{n-1}(x) -c_n K_{n-2}(x),     $$
where $K_0 =1$ and $c_n >0$ for all $n.$   Then the sequence $(K_n(x))_{n=0}^\infty$ 
forms a sequence of orthogonal polynomials with respect to some positive measure on the 
real line.}

Combining (\ref{f76}) with Favard's theorem, we conclude  that the sequence $(M_{n, a}(x))_{n=0}^\infty$ 
forms a sequence of orthogonal polynomials with respect to some positive measure on the 
real line. 

A fundamental property of orthogonal polynomial sequences is that all their zeros are real and simple, 
and that the zeros of consecutive polynomials interlace. Thus, all zeros of the polynomials $M_{n, a}$ 
are real and simple,  and zeros of consecutive polynomials interlace. Hence, all the zeros of the 
polynomials $S_{n, a}$ are simple and purely imaginary,  and zeros of consecutive polynomials interlace. 
Equivalently, all zeros of the polynomials  $P_{n, a}$ are simple and have real parts equal to
$- \frac{a}{2},$  and zeros of consecutive polynomials interlace.

Theorem~\ref{th1.0} is proved.
\end{proof}

\end{Example}

\section{Open Problems}

In connection with the examples considered, we formulate the following open problems.

\begin{Problem}
\label{pr4}
Let $a>0$ and
\begin{eqnarray}
\label{f13}
& U_n = \{P\in \mathbb{R}[x]  :  \deg P =n\  \mbox{and} 
\\ \nonumber &    
\mathcal{L}_a(P) \   \mbox{ has all real zeros 
in the segment }\  [-1, 0]  \}. 
\end{eqnarray}

For $n\in \mathbb{N},$ define $r_n$  by
\begin{equation}
\label{f13}
r_n = \sup_{P\in U_n} (\min\{|z| \in \mathbb{C} : P(z)=0 \}).
\end{equation}
The problem is to describe (or estimate) the  sequence 
$(r_n)_{n=1}^\infty.$ Obviously, $r_1=a.$  By Example~\ref{ex2.1} , 
 $r_2 \geq \frac{a+\sqrt{a(a+1)} }{2} >a.$ Example~\ref{ex3} and 
Statement 4 of Theorem~\ref{th1} show that $r_{2m+1}=a, m \in \mathbb{N}.$ 
Our conjecture is that the sequence $\{r_n\}_{n=1}^\infty$ is bounded from above.
\end{Problem}

\begin{Problem}
\label{pr5}
For $n\in \mathbb{N},$  define   $R_n$ by
\begin{equation}
\label{f13}
R_n = \sup_{P\in U_n} (\max\{|z| \in \mathbb{C} : P(z)=0 \}).
\end{equation}
The problem is to describe (or estimate)  the sequence 
$(R_n)_{n=1}^\infty.$ Obviously, $R_1=r_1=a.$ Example~\ref{ex3}
shows that $R_n \geq a+n-1.$ Note that item 1 of  Statement~\ref{st1} and 
formula~(\ref{fd2}) show that all real zeros of a real polynomial $P\in U_n$  
lie in the segment $[-(a+n-1), n-1].$   Our conjecture is that $R_n = a+n-1.$ 
\end{Problem}

\begin{Problem}
\label{pr6}
Let $P = C(x-x_1)(x-x_2) \cdot \ldots \cdot (x-x_n)\in \mathbb{R}[x],$  
$x_1 \leq x_2 \leq \ldots \leq x_n,$  be a polynomial such that  $\mathcal{L}_a(P) $   has all real zeros 
in the segment  $  [-1, 0].$ 
Is it true that there exists a constant $d(a)\geq 1,$ independent of n, such that for every 
$j=1, 2, \ldots, n-1,$ we have $x_{j+1} - x_j \leq d(a)?$   
Note that for every polynomial $P$ satisfying the sufficient conditions of 
Theorem~B, the statement is true with $d(a)=a.$ 
\end{Problem}

\end{document}